\documentclass[letterpaper, 10 pt, conference]{ieeeconf}  
\IEEEoverridecommandlockouts                              
\usepackage[english]{babel}
\usepackage{amssymb}
\usepackage{amsmath}
\usepackage{graphicx}
\usepackage{subfigure}
\usepackage{xcolor}
\usepackage{tikz}
\usepackage{cleveref}
\usepackage{pgfplots}

\usepackage{amsthm}
\usepackage{cite}

\newtheorem{assumption}{Assumption}
\newtheorem{theorem}{Theorem}
\newtheorem*{sassumption}{Standing Assumption}
\newtheorem{definition}{Definition}
\newtheorem{proposition}{Proposition}

\newtheorem{corollary}{Corollary}

\newcommand{\be}{\begin{equation}}
\newcommand{\ee}{\end{equation}}

\newcommand{\E}{\mathbb E}
\newcommand{\Real}[1]{ { {\mathbb R}^{#1} } }

\def\myspace{1.5mm}
\DeclareMathOperator*{\argmin}{arg\,min}
\newcommand{\md}{\mathbf{d}}
\newcommand{\xstar}{x^\star}
\newcommand{\xne}{x_{\rm NE}}
\newcommand{\sstar}{s^\star}
\newcommand{\mb}[1]{\mathbb{#1}}
\newcommand{\mc}[1]{\mathcal{#1}}

\newcommand{\N}{M}
\newcommand{\X}{\mc{X}}
\renewcommand{\j}{j}
\newcommand{\xj}{x^{\j}}
\newcommand{\Xj}{\X^{\j}}
\newcommand{\Xjdelta}{\X^{\j}_{\delta}}
\newcommand{\Xjdeltai}{\X^{\j}_{\delta_i}}
\newcommand{\n}{m}
\newcommand{\Jj}{J^{\j}}
\newcommand{\Jjmax}{J_{\rm max}^{\j}}

\newcommand{\xre}{x_{\rm R}}
\newcommand{\xsr}{x_{\rm SR}}
\newcommand{\tsr}{t_{\rm SR}}
\newcommand{\ysr}{y_{\rm SR}}
\newcommand{\noin}{\noindent}
\usepackage{amsfonts}
\DeclareSymbolFont{bbold}{U}{bbold}{m}{n}
\DeclareSymbolFontAlphabet{\mathbbold}{bbold}
\newcommand{\vect}[1]{\mathbbold{#1}}
\newcommand{\zeros}[1]{\vect{0}_{#1}}
\newcommand{\ones}[1]{\vect{1}_{#1}}
\newcommand{\mcm}{\mc{\N}}%
\title{\LARGE \bf 
The scenario approach meets uncertain variational inequalities and game theory
}

\author{Dario Paccagnan and Marco C. Campi%
\thanks{This work was supported by the SNSF grant $\#$P2EZP2 181618 and by the University of Brescia under the project CLAFITE. 
D. Paccagnan is with the Mechanical Engineering Department and the Center of Control, Dynamical Systems and Computation, UC Santa Barbara, CA 93106-5070, USA.
M. Campi is with Dipartimento di Ingegneria dell'Informazione, Universit\`a di Brescia, Italy.
Email: {\tt dariop@ucsb.edu}, {\tt marco.campi@unibs.it}
}}

\begin{document}
\maketitle
\thispagestyle{empty}
\pagestyle{empty}

\begin{abstract}
Variational inequalities are modelling tools used to capture a variety of decision-making problems arising in mathematical optimization, operations research, game theory.
The scenario approach is a set of techniques developed to tackle stochastic optimization problems, take decisions based on historical data, and quantify their risk.
The overarching goal of this manuscript is to bridge these two areas of research, and thus broaden the class of problems amenable to be studied under the lens of the scenario approach.
First and foremost, we provide out-of-samples feasibility guarantees for the solution of variational and quasi variational inequality problems.
Second, we apply these results to two classes of uncertain games.
 In the first class, the uncertainty enters in the constraint sets, while in the second class the uncertainty enters in the cost functions.
Finally, we exemplify the quality and relevance of our bounds through numerical simulations on a demand-response model.

\end{abstract}

\section{Introduction}
\label{}

Variational inequalities are a very rich class of %
decision-making problems. %
They can be used, for example, to characterize the solution of a convex optimization program, or to capture the notion of saddle point in a min-max problem.
Variational inequalities can also be employed to describe complementarity conditions, nonlinear systems of equations, or equilibrium notions such as that of Nash or Wardrop equilibrium \cite{facchinei2007finite}. With respect to the applications, variational inequalities have been employed in countless fields,  including transportation networks, demand-response markets,  option pricing, \mbox{structural analysis, evolutionary biology \cite{dafermos1980traffic, gentile2017nash, jaillet1990variational, ferris2001limit, sandholm2010population}.}

Many of  these settings feature a non-negligible source of uncertainty, so that any 
	planned action
inevitably comes with a 
	degree of risk. 
While deterministic models have been widely used as a first order approximation, the increasing availability of raw data motivates the development of data-based techniques for decision-making problems, amongst which variational inequalities are an important class. 
As a concrete example, consider that of drivers moving on a road traffic network with the objective of reaching their destination as swiftly as possible.
Based on historical data%
, a given user would like to i) plan her route, and ii)  estimate how likely she is to reach the destination within a given time.

Towards this goal, it is natural to consider variational inequalities where the solution is required to be robust against a set of observed realizations of the uncertainty, as formalized next. %
Given a collection of sets $\{\mc{X}_{\delta_i}\}_{i=1}^N$, where $\{\delta_i\}_{i=1}^N$ are independent observations from the probability space $(\Delta,\mc{F},\mb{P})$, and given $F: \Real{n}\rightarrow \Real{n}$, we consider the following variational inequality problem:
\be%
\label{eq:mainproblem}
\text{find}~~\xstar\in\mc{X}\doteq \bigcap_{i=1}^N \mathcal{X}_{\delta_i} ~~\text{s.t.}~~ 
F(x^\star)^\top (x - x^\star) \ge 0~~
\forall x \in\mc{X}.
\ee
We assume no information is available on $\mb{P}$, and ask the following fundamental question: 
\emph{how likely is a solution of \eqref{eq:mainproblem} to be robust against unseen realizations?}

In this respect, our main objective is to provide probabilistic bounds on the feasibility of a solution to \eqref{eq:mainproblem}, while ensuring that such solution can be computed using a tractable algorithm.
While our results are de facto probabilistic feasibility statements,  we will show how to apply them to game theoretic models to e.g., quantify the probability of incurring a higher cost compared to what originally predicted.\footnote{This is the main reason to extend the results derived for \eqref{eq:mainproblem} to the richer class of quasi variational inequality problems, see \Cref{subsec:quasivi,sec:robustgames}.} 

\vspace*{\myspace}
\noin {\bf Related works.}
Two formulations are typically employed to incorporate uncertainty into variational inequality models \cite{shanbhag2013stochastic}. A first approach, termed \emph{expected-value formulation}, captures uncertainty arising in the corresponding operator $F$ in an average sense. Given $\mc{X}\subseteq\Real{n}$, $F:\mc{X}\times\Delta\rightarrow \Real{n}$, and a probability space $(\Delta,\mc{F},\mb{P})$, a solution to the expected-value variational inequality is an element $\xstar\in\mc{X}$ such that 
\be
\label{eq:expectedVI}
\E[F(\xstar,\delta)]^\top (x-\xstar)\ge0 \qquad\forall x\in\mc{X}.
\ee
Naturally, if the expectation can be easily evaluated, solving \eqref{eq:expectedVI} is no harder than solving a deterministic variational inequality, for which much is known (e.g., existence and uniqueness results, algorithms \cite{facchinei2007finite}). If this is not the case, one could employ sampling-based algorithms to compute an approximate solution of \eqref{eq:expectedVI}, see \cite{gurkan1999sample,jiang2008stochastic, yousefian2018stochastic}.

A second approach, which we refer to as the \emph{robust formulation}, is used to accommodate uncertainty both in the operator, and in the constraint sets. Consider the collection $\{\mc{X}_\delta\}_{\delta\in\Delta}$, where $\mc{X}_\delta\subseteq\Real{n}$, and let $\mc{X}\doteq\cap_{\delta\in\Delta} \mc{X}_{\delta}$. A solution to the robust variational inequality is an element $\xstar\in\mc{X}$ s.t.
\be
\label{eq:almostsureVI}
F(\xstar,\delta)^\top (x-\xstar)\ge0 \qquad\forall x\in\mc{X}, \quad \forall \delta \in\Delta.
\ee
It is worth noting that, even when the uncertainty enters only in $F$, %
a solution to \eqref{eq:almostsureVI} is unlikely to exists.\footnote{To understand this, consider the case when the variational inequality is used to describe the first order condition of a convex optimization program. Within this setting, \eqref{eq:almostsureVI} requires $x^\star$ to solve a \emph{family} of different optimization problems, one per each $\delta\in\Delta$. Thus, \eqref{eq:almostsureVI} only exceptionally has a solution.}
The above requirement is hence weakened employing a formulation termed \emph{expected residual minimization} (ERM), see \cite{chen2005expected}. Within this setting, given a probability space $(\Delta,\mc{F},\mb{P})$, a solution is defined as
$
x^\star\in\arg\min_{x\in\mc{X}}\E[\Phi(x,\delta)],
$
where $\Phi:\mc{X}\times\Delta\rightarrow\mb{R}$ is a residual function.\footnote{A function $\Phi:\mc{X}\times\Delta\rightarrow\mb{R}$ is a residual function if, $\Phi(x,\delta)=0$ whenever $x$ is a solution of \eqref{eq:almostsureVI} for given $\delta$, and $\Phi(x,\delta)>0$ elsewhere.}
In other words, we look for a point that satisfies \eqref{eq:almostsureVI} as best we can (measured through $\Phi$), on average over $\Delta$.
Sample-based algorithms for its approximate solution are derived in, e.g., \cite{chen2012stochastic,luo2009expected}.

While the subject of our studies, defined in \eqref{eq:mainproblem}, differs in form and spirit from that of \eqref{eq:expectedVI}, it can be regarded as connected to \eqref{eq:almostsureVI}. Indeed, our model can be thought of as a sampled version of \eqref{eq:almostsureVI}, where the uncertainty enters only in the constraints. 
In spite of that, our objectives significantly depart from that of the ERM formulation, as detailed next.

\vspace*{\myspace}
\noin {\bf Contributions.} 
The goal of this manuscript is that of \emph{quantifying the risk} associated with a solution of \eqref{eq:mainproblem} against unseen samples $\delta\in\Delta$, while ensuring that such solution can be computed tractably. Our main contributions are as follows.

\begin{enumerate}
\item[i)] We provide a-priori and a-posteriori bounds on the probability that the solution of \eqref{eq:mainproblem} remains feasible for unseen values of $\delta\in\Delta$ (out-of-sample guarantees).
\item[ii)]	We show that the bounds derived in i) hold for the broader class of \emph{quasi variational inequality} problems.

\item[iii)] We leverage the bounds obtained in i) to study Nash equilibrium problems with uncertain constraint sets.
\item[iv)] We employ the bounds derived in ii) to give concrete probabilistic guarantees on the performance of Nash equilibria, relative to games with uncertain payoffs, as originally defined by Aghassi and Bertsimas in \cite{aghassi2006robust}.
\item[v)] We consider a simple demand-response scheme %
and exemplify the applicability and quality of our probabilistic bounds through numerical simulations.
\end{enumerate}

Our results follow the same spirit of those derived within the so-called \emph{scenario approach}, where the sampled counterpart of a robust optimization program is considered, and the risk  %
associated to the corresponding solution is bounded in a probabilistic sense \cite{calafiore2005uncertain,campi2008exact,margellos2014road,esfahani2015performance,care2015scenario,care2018new,campi2018wait}. To the best of the authors' knowledge, our contribution is the first to enlarge the applicability of the scenario approach to the broader class of variational inequality problems.
While variational inequalities are used to model a wide spectrum of problems, we limit ourselves to discuss the impact of our results on the class of uncertain games, due to space considerations.

\vspace*{\myspace}
\noin {\bf Organization.}
In \Cref{sec:scenarioandvi} we introduce the main subject of our analysis, as well as some preliminary notions. \Cref{sec:mainresult} contains the main result and its extension to quasi variational inequalities.
In \Cref{sec:robustgames} we show the relevance of the bounds previously derived in connection to uncertain games.
In \Cref{sec:numerics} we test our results on a demand-response scheme through exhaustive numerical simulations.

\section{The scenario approach to variational inequalities}
\label{sec:scenarioandvi}
Motivated by the previous discussion, in the remainder of this paper we consider the variational inequality (VI) introduced in \eqref{eq:mainproblem} and reported in the following: 
\[
\text{find}~~\xstar\in\mc{X}\doteq \bigcap_{i=1}^N \mathcal{X}_{\delta_i} ~~\text{s.t.}~~ 
F(x^\star)^\top (x - x^\star) \ge 0~~
\forall x \in\mc{X},
\]
where $F:\Real{n}\rightarrow\Real{n}$ and $\mc{X}_{\delta_i}\subseteq\Real{n}$ for $i\in\{1,\dots,N\}$ are elements of a family of sets $\{\mc{X}_\delta\}_{\delta\in\Delta}$. 
Throughout the presentation we assume that $\{\delta_i\}_{i=1}^N$ are independent samples from the probability space $(\Delta,\mc{F},\mb{P})$, though no 
	knowledge is assumed on $\mb{P}$.
In order to provide out-of-sample guarantees on the feasibility of a solution to \eqref{eq:mainproblem}, we begin by introducing two concepts that play a key role: the notion of risk and that of support constraint.

\begin{definition}[Risk]
\label{def:risk}
	The \emph{risk} of a given $x\in\mc{X}$ is given by
	$
	V(x)\doteq\mb{P}\{\delta\in\Delta \text{ s.t. } x\notin\mc{X}_\delta\}.
	$
\end{definition}

\noin The quantity $V(x)$ measures the violation of the constraints defined by $x\in\mc{X}_\delta$ for all $\delta\in\Delta$. As such, $V:\mc{X}\rightarrow [0,1]$ and, for fixed $x$, it constitutes a deterministic quantity. Nevertheless, since $x^\star$ is a random variable (through its dependance on $(\delta_1,\dots,\delta_N)$), the risk $V(x^\star)$ associated with the solution $x^\star$ is also a \emph{random variable}.\footnote{We assume measurability of all the quantities introduced in this paper.} Our objective will be that of acquiring deeper insight into its distribution. 
\begin{sassumption}[Existence and uniqueness]
\label{ass:exun}
For any $N$ and for any tuple $(\delta_1,\dots,\delta_N)$, the variational inequality \eqref{eq:mainproblem} admits a unique solution identified with $x^\star$. 
\end{sassumption}
\noin 
Throughout the manuscript, we assume that the Standing Assumption is satisfied, so that $x^\star$ is well defined and unique.
It is worth noting that the existence of a solution to \eqref{eq:mainproblem} is guaranteed under very mild conditions on the operator $F$ and on the constraints set $\mc{X}$. Uniqueness of $x^\star$ is instead obtained under structural assumptions on $F$ (e.g., strong monotonicity). While these cases do not encompass all possible variational inequalities arising from \eqref{eq:mainproblem}, the set-up is truly rich and includes important applications such as traffic dispatch \cite{dafermos1980traffic}, cognitive radio systems \cite{scutari2012monotone}, demand-response markets \cite{gentile2017nash}, and many more. Sufficient conditions guaranteeing the satisfaction of the Standing Assumption are presented in \Cref{prop:suffexun}, included at the end of this section.

\begin{definition}[Support constraint]
\label{def:support}
A constraint $x\in\mc{X}_{\delta_i}$ is of support for \eqref{eq:mainproblem}, if its removal modifies the solution $x^\star$. We denote with $S^\star$ the set of \mbox{support constraints associated to $\xstar$.}
\end{definition}
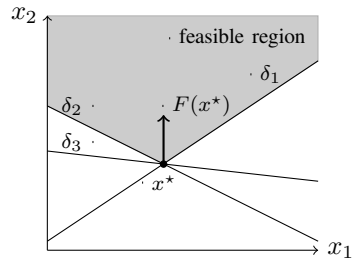
\begin{figure}[h!] 
\vspace*{-5mm}
\begin{center}
\begin{tikzpicture}[scale=0.6]
 
\draw[->] (0,0.8) -- (6,0.8) node[anchor=west]{$x_1$};
\draw[->] (0,0.8) -- (0,6) node[anchor=east]{$x_2$};

\draw (0,1) -- (6,5) node[anchor=north east]{};
\draw (6,1) -- (0,4) node[anchor=north east]{};
\draw (6,2.33) -- (0,3) node[anchor=north east]{};
                  
\filldraw[draw=black,%
          fill=black,
          opacity=0.2,
         ]
           	   (2.575,2.71)
           	-- (6,5)
            -- (6,6)
            -- (0,6)
            -- (0,4)
            -- (2.575,2.71)
            -- cycle;
\filldraw (2.575,2.71) circle (2pt);

\draw[thick, dashed] (2.7,5.5) -- (2.7,5.5) node[anchor=west]{\footnotesize feasible region}; 

\draw[thick, dashed] (1,4) -- (1,4) node[anchor=east]{\footnotesize $\delta_2$};

\draw[thick, dashed] (1,3.2) -- (1,3.2) node[anchor=east]{\footnotesize $\delta_3$};

\draw[thick, dashed] (4.5,4.7) -- (4.5,4.7) node[anchor=west]{\footnotesize $\delta_1$};

\draw[thick, dashed] (2.1,2.3) -- (2.1,2.3) node[anchor=west]{\footnotesize $x^\star$};

\draw[thick, ->] (2.575,2.71) -- (2.575,3.8) node[anchor=north east]{};

\draw[thick, dashed] (2.55,4) -- (2.55,4) node[anchor=west]{\footnotesize $F(x^\star)$};  

\end{tikzpicture} 
\vspace*{-3mm}
\caption{An example of variational inequality \eqref{eq:mainproblem} in dimension two. Each sample in $\{\delta_i\}_{i=1}^3$ defines a feasible region, and the grey area describes the set $\mc{X}$. Note that $\xstar$ is a solution of \eqref{eq:mainproblem} as $\xstar\in\mc{X}$, and the inner product between $F(\xstar)$ with any feasible direction at $\xstar$ is non-negative. Finally, observe that $\delta_1$ defines a support constraint, while $\delta_2$ or $\delta_3$ do not.}
\label{fig:degenerate}
\end{center}
\vspace*{-5mm}
\end{figure}

\noin Within the example depicted in \Cref{fig:degenerate}, it is worth noting that the removal of the constraints $\mc{X}_{\delta_2}$ or $\mc{X}_{\delta_3}$ - \emph{one at a time} - does not modify the solution: indeed neither $\mc{X}_{\delta_2}$, nor $\mc{X}_{\delta_3}$ are support constraints. Nevertheless, the \emph{simultaneous} removal of both $\mc{X}_{\delta_2}$ and $\mc{X}_{\delta_3}$ does change the solution.
To rule out degenerate conditions such as this, we introduce the following assumption, adapted from \cite[Ass. 2]{campi2018wait}.\footnote{If $\{\mc{X}_{\delta_i}\}_{i=1}^N$ are convex, the degenerate instances constitute exceptional situations in that they require the constraints to accumulate precisely at the solution $\xstar$, as in Figure \ref{fig:degenerate}. On the contrary, in the case of non-convex constraint sets, degenerate cases are much more common, see \cite[Sec. 8]{campi2018wait}.}

\begin{assumption}[Non-degeneracy]
\label{ass:nondegeneracy}
 The solution $x^\star$ coincides $\mb{P}^N$ almost-surely %
 with the solution obtained by eliminating all the constraints that are not of support.
\end{assumption}

 We conclude this section providing sufficient conditions that guarantee the existence and uniqueness of the solution to \eqref{eq:mainproblem}, so that the Standing Assumption holds.
\begin{proposition}[Existence and uniqueness, \textup{\cite%
{facchinei2007finite}}]
\label{prop:suffexun}
\begin{itemize}
\item[]
	\item[-] If $\mc{X}$ is nonempty compact convex, and $F$ is continuous, then the solution set of \eqref{eq:mainproblem} is nonempty and compact.
	\item[-]  If $\mc{X}$ is nonempty closed convex, and $F$ is strongly monotone on $\mc{X}$,  then \eqref{eq:mainproblem} admits a unique solution.\footnote{\label{foot:SMON}An operator $F:\mc{X}\rightarrow\Real{n}$ is strongly monotone on $\mc{X}$ if there exists $\alpha>0$ such that $(F(x)-F(y))^\top(x-y)\ge\alpha||x-y||^2$ for all $x,y\in\mc{X}$. If $F$ is continuously differentiable, a sufficient (and easily checkable) condition amounts to requiring the Jacobian of $F$ to be uniformly positive definite, that is $y^\top JF(x)y\ge \alpha||y||^2$ for all $y\in\Real{n}$ for all $x\in\mc{X}^{\rm o}$, where $\mc{X}^{\rm o}$ is an open superset of $\mc{X}$, see \cite[Prop. 2.3.2]{facchinei2007finite}.}
\end{itemize}
\end{proposition}
\noin
If \eqref{eq:mainproblem} is used to characterize the solution of a strongly convex and smooth optimization program (i.e. if $F(x)=\nabla_xJ(x)$ for a smooth $J:\mc{X}\rightarrow\mb{R}$), the previous proposition applies directly since the gradient of a strongly convex function is strongly monotone~\cite[Prop. 17.10]{bauschke2011convex}.

\section{Main result: probabilistic feasibility for variational inequalities}
\label{sec:mainresult}
The aim of this section is to provide bounds on the risk associated to the solution of \eqref{eq:mainproblem} that hold with high confidence.  \mbox{Towards this goal, we introduce the map $t:\mb{N}\rightarrow [0,1]$.}%
\begin{definition}
\label{def:t}
Given $\beta \in(0,1)$, for any $k\in\{0,\dots,N-1\}$ consider the polynomial equation in the unknown $t$ 
\be
\frac{\beta}{N+1}\sum_{l=k}^{N}\binom{l}{k} t^{l-k}
-\binom{N}{k}t^{N-k}=0.
\label{eq:poly}
\ee
Let $t(k)$ be its unique solution in the interval $(0,1)$.\footnote{Existence and uniqueness of the solution to the polynomial equation \eqref{eq:poly} is shown in \cite[Thm. 2]{campi2018wait}.} 
Further, let $t(k)=0$ for any $k\ge N$.
 \end{definition}
\noin 
As shown in recent results on scenario optimization \cite{campi2018wait}, the distribution of the risk $V(x^\star)$ is intimately connected with the number of support constraints at the solution $x^\star$, which we identify with $\sstar=|S^\star|$. 
Given a confidence parameter $\beta\in(0,1)$, our goal is to determine a function $\varepsilon(\sstar)$ so that
\[
{\mb{P}}^N[V(\xstar)\le \varepsilon(\sstar)]\ge 1-\beta
\]
 holds true for any variational inequality \eqref{eq:mainproblem} satisfying the Standing Assumption and Assumption \ref{ass:nondegeneracy}. Theorem \ref{thm:mainresult} provides one way to construct $\varepsilon(\sstar)$ by means of $t(\sstar)$.
\begin{theorem}[Probabilistic feasibility for VI]
\label{thm:mainresult}
Given $\beta \in (0,1)$, consider $t:\mb{N}\rightarrow\mb{R}$ as per Definition \ref{def:t}.
\begin{itemize}
\item[(i)]
Under the Standing Assumption, and Assumption \ref{ass:nondegeneracy}, for any $\Delta$ and $\mb{P}$ it holds that
\be
\label{eq:feas_general}
{\mb{P}}^N[V(\xstar)\le \varepsilon(\sstar)]\ge 1-\beta
\quad \text{with} \quad
\varepsilon(\sstar) \doteq 1- t(\sstar),
\ee
\item[(ii)] If, in addition, the constraint sets $\{\mc{X}_{\delta_i}\}_{i=1}^N$ are convex, then $\sstar\le n$ (dimension of the decision variable $x$), and the following a-priori bound holds for all $\Delta$ and $\mb{P}$
\[
{\mb{P}}^N[V(\xstar)\le \varepsilon(n)]\ge 1-\beta
\quad \text{with} \quad
\varepsilon(n) \doteq 1- t(n).
\]
\end{itemize}
\end{theorem}

\begin{proof}
The proof is deferred to the Appendix.
\end{proof}

The first statement in Theorem \ref{thm:mainresult} provides an a-posteriori bound, and requires no additional assumption other than the Standing Assumption and Assumption \ref{ass:nondegeneracy} (e.g., no convexity of the constraint sets is required). In practice, one computes a solution to \eqref{eq:mainproblem}, determines $\sstar$, and is then given a probabilistic feasibility statement for any choice of $\beta\in(0,1)$.\footnote{Computing the number of support constraints can be easily achieved by solving the original problem where constraints are removed one at a time.} In this respect, we are typically interested in selecting $\beta$ very small (e.g., $10^{-6}$) so that the statement $V(\xstar)\le \varepsilon(\sstar)$ holds with very high confidence (e.g., $1-10^{-6}=0.999999$). 
Upon assuming convexity of the constraints sets, the second statement provides an a-priori bound of the form \eqref{eq:feas_general} where $\sstar$ is replaced by $n$ (the dimension of the decision variable). Overall, Theorem \ref{thm:mainresult} shows that the upper bound on the risk derived in \cite[Thm. 4]{campi2018wait} is not limited to optimization programs, but holds for the far more general class of variational inequality problems. For a plot of $\varepsilon(k)$, see \cite[Fig. 3]{campi2018wait}.

\vspace*{\myspace}
\noin {\bf Computational aspects.} While Theorem~\ref{thm:mainresult} provides certificates of probabilistic feasibility, its result is of practical interest especially if it is possible to determine a solution of \eqref{eq:mainproblem} \emph{efficiently}. With respect to the computational aspects, much is known for the class of monotone variational inequalities, i.e. those variational inequalities where the operator $F$ is monotone or strongly monotone (see \cref{foot:SMON} for a definition). Examples of efficient algorithms for strongly monotone and monotone variational inequalities include projection methods, proximal methods, splitting and interior point methods.\footnote{We redirect the reader to \cite[Chap. 12]{facchinei2007finite} for an extensive treatment.} On the contrary, if the operator associated to \eqref{eq:mainproblem} is \emph{not monotone}, the problem is intractable to solve in the worst-case. 
Indeed, non-monotone variational inequalities hold non-monotone linear complementarity problems as a special case. The latter class is known \mbox{to be $\mc{NP}$-complete \cite{chung1989np}.}

\subsection{Extension to quasi variational inequalities}
\label{subsec:quasivi}
In this section we show how the results of Theorem \ref{thm:mainresult} carry over to the case when \eqref{eq:mainproblem} is replaced by a more general class of problems known as quasi variational inequality (QVI). Informally, quasi variational inequalities extend the notion of variational inequality by allowing the decision set $\mc{X}$ to be parametrized by $x$, see  \cite{chan1982generalized}. QVIs are important tools used to model complex equilibrium problems arising in various fields such as games with shared constraints, transportation network, solid mechanics, biology, and many more \cite{bensoussan1984controle, beremlijski2002shape, bliemer2003quasi, hammerstein1994game, kravchuk2007variational}.
As we shall see in Section \ref{sec:robustgames}, this generalization will be used to provide concrete performance guarantees for robust Nash equilibrium problems, whenever the uncertainty enters in the agents' cost functions. Let $\mc{X}_{\delta_i} : \mb{R}^n \rightrightarrows 2^{\mb{R}^n}$ be elements of a collection of set-valued maps 
$\{\mc{X}_\delta\}_{\delta\in\Delta}$, for $i\in\{1,\dots,N\}$.
Given $F:\mb{R}^n\rightarrow\mb{R}^n$, we consider the following quasi variational inequality problem: find $\xstar\in\mc{X}(\xstar) \doteq\cap_{i=1}^N \mc{X}_{\delta_i}(\xstar)$ such that
\be
\label{eq:QVI}
F(x^\star)^\top (x - \xstar) \ge 0
\qquad
\forall x \in\mc{X}(\xstar).
\ee 
Once more, we assume that $\{\delta_i\}_{i=1}^N$ are independent samples from the probability space $(\Delta,\mc{F},\mb{P})$.
Additionally, we assume that \eqref{eq:QVI} admits a unique solution.
The notion of support constraint carries over unchanged from Definition \ref{def:support}, while the notion of risk requires a minor adaptation. %
\begin{definition}[Risk for QVI]
\label{def:riskQVI}
	The \emph{risk} associated to $x$ is 
	\[
	V(x)\doteq\mb{P}\{\delta \in \Delta~\text{s.t.}~x\notin\mc{X}_\delta(x)\}
	\]
\end{definition}
\noin
The next theorem shows that the main result presented in Theorem \ref{thm:mainresult} extends to quasi variational inequalities.
\begin{theorem}[Probabilistic feasibility for QVI]
\label{thm:mainresultQVI}
Let $\xstar$ be the (unique) solution of \eqref{eq:QVI} and $\sstar$ be the number of support constraints.
Let Assumption \ref{ass:nondegeneracy} hold.
Given $\beta\in(0,1)$, let $t:\mb{N}\rightarrow\mb{R}$ be as per \Cref{def:t}. Then, for any $\Delta$, $\mathbb{P}$, it is 
\be
\label{eq:QVIbound}
{\mb{P}}^N[V(\xstar)\le \varepsilon(\sstar)]\ge 1-\beta
\quad \text{where} \quad
\varepsilon(\sstar) \doteq 1- t(\sstar).
\ee
If, in addition, the sets $\{\mc{X}_{\delta_i}(\xstar)\}_{i=1}^N$ are convex, then $\sstar \le n$ and the bound \eqref{eq:QVIbound} holds a-priori with $n$ in place of $\sstar$.
\end{theorem}
\begin{proof}
The proof is omitted, due to space considerations. Nevertheless, it is possible to follow (mutatis mutandis) the derivation presented in the proof of \Cref{thm:mainresult}.	
\end{proof}
\section{Application to robust game theory}
\label{sec:robustgames}
\subsection{Uncertainty entering in the constraint sets}
\label{subsec:uncertainconstr}
We begin by considering a general game-theoretic model, where agents aim to minimize private cost functions, while satisfying uncertain local constraints \emph{robustly}. %
Formally, each agent $j\in\mcm=\{1,\dots,\N\}$ is allowed to select $\xj\in\Xj\doteq \cap_{i=1}^N \Xjdeltai\subseteq \mb{R}^\n$, where $\{\Xjdeltai\}_{i=1}^N$ is a collection of sets from the family $\{\Xjdelta\}_{\delta \in \Delta}$, and $\{\delta_i\}_{i=1}^N$ are independent samples from the probability space $(\Delta,\mc{F},\mb{P})$. %
Agent $\j\in\mcm$ aims at minimizing the cost function $\Jj: \Xj\rightarrow\mb{R}$. 
To ease the notation, we define $x^{-\j}=(x^1,\dots,x^{\j-1},x^{\j+1},\dots,x^\N)$, for any $j\in\mcm$.
We consider the notion of Nash equilibrium.
\begin{definition}[Nash equilibrium]
A tuple $\xne=(\xne^1,\dots,\xne^\N)$ is a Nash equilibrium if $\xne\in\X^1\times\dots\times \X^{\N}$ and $\Jj(\xne^{\j}, \xne^{-\j})\le \Jj(\xj,\xne^{-\j})$ for all deviations $\xj\in\Xj$ and for all agents $j\in\mcm$.
\end{definition}
\begin{assumption}
\label{ass:nashexists}
For all $\j\in\mcm$, the cost function $\Jj$ is continuously differentiable, and convex in $\xj$ for any fixed $x^{-j}$. The sets $\{\Xj\}_{j=1}^\N$ are non-empty, closed, convex for every tuple $(\delta_1,\dots,\delta_N)$, for every $N$.
\end{assumption}
The next proposition, adapted from \cite{facchinei2007finite} draws the key connection between Nash equilibria and variational inequalities.
\begin{proposition}[Nash equilibria and VI \text{\cite[Prop. 1.4.2]{facchinei2007finite}}]
\label{prop:NEandVI}
Let Assumption \ref{ass:nashexists} hold. Then a point $\xne$ is a Nash equilibrium if and only if it solves \eqref{eq:mainproblem}, with 
\be
\label{eq:defFX}
F(x)\doteq
\begin{bmatrix}
\nabla_{x^1} J^1(x)\\
\vdots\\
\nabla_{x^{\N}} J^{\N}(x)\\
\end{bmatrix},
\quad
\mc{X}_{\delta_i}
\doteq
\X^1_{\delta_i}\times\dots\times\X^\N_{\delta_i}.
\ee
\end{proposition}
\begin{proof}
The proof is reported in the Appendix.
\end{proof}

Within the previous model, the uncertainty described by $\delta\in\Delta$ is meant as \emph{shared} among the agents. This is indeed the most common and challenging situation. In spite of that, our model also includes the case of \emph{non-shared} uncertainty, i.e. the case where $\Xjdelta$ is of the form $\mc{X}^j_{\delta^j}$ as $\delta$ can represent a vector of uncertainty. Limitedly to the latter case, it is possible to derive probabilistic guarantees on each agent's feasibility by direct application of the scenario approach \cite{campi2008exact} to each agent optimization 
$
\xne^j\in\argmin_{\xj\in\mc{X}^\j}\Jj(\xj,\xne^{-\j})
$
, after having fixed $x^{-\j}=\xne^{-\j}$.
\noin Nevertheless, for the case of shared uncertainty, a direct application of \cite{campi2018wait} provides no answer.\footnote{To see this, observe that a constraint that is not of support for agent's $j$ optimization program with fixed $x^{-\j}=\xne^{-\j}$, might instead be of support for the Nash equilibrium problem, as its removal could modify $\xne^{-\j}$, which in turn modifies $\xne^{\j}$.} 
Instead, the following corollary offers probabilistic feasibility guarantees for $\xne$. In this context, a constraint is of support for the Nash equilibrium problem, if its removal changes the solution.
\begin{corollary}[Probabilistic feasibility for $\xne$]
\label{cor:first}
\begin{itemize}
\item[]
\item[-] Let Assumption \ref{ass:nashexists} hold. Then, a Nash equilibrium exists.
\item[-] Further assume that the operator $F$ defined in \eqref{eq:defFX} is strongly monotone. Then, $\xne$ is unique.
\item[-] Fix $\beta\in (0,1)$, and let $t:\mb{N}\rightarrow\mb{R}$ be as per Definition \ref{def:t}.
In addition to the previous assumptions, assume that $\xne$ coincides $\mb{P}^N$ almost-surely with the Nash equilibrium of a game obtained by eliminating all the constraints that are not of support.
Then, the following a-posteriori and a-priori bounds hold for any $\Delta$ and $\mb{P}$ 
\[
\begin{split}
&\mb{P}^N [V(\xne)\le\varepsilon(\sstar)]\ge 1-\beta
~\text{with}~
\varepsilon(\sstar) \doteq 1- t(\sstar),\\
&\mb{P}^N [V(\xne)\le\varepsilon(n)]\ge 1-\beta ~~\text{with}~
\varepsilon(n) \doteq 1- t(n),
\end{split}
\]
where $n$ is the dimension of the decision variable $x$, and $\sstar$ is the number support constraints of $\xne$.
\end{itemize}
\end{corollary}
\begin{proof}
See the Appendix.	
\end{proof}
\noin
A consequence of \Cref{cor:first} is the possibility to bound the infeasibility risk associated to any agent $\j\in\mcm$. Indeed, let 
$
V^j(x)\doteq\mb{P}\{\delta \in \Delta~\text{s.t.}~\xj\notin\mc{X}^j\}
$. Since $V^j(x)\le V(x)$, \Cref{cor:first} ensures that $\mb{P}^N [V^j(\xne)\le\varepsilon(\sstar)]\ge 1-\beta$.

\subsection{Uncertainty entering in the cost functions}
We consider a game-theoretic model where the cost function associated to each agent depends on an uncertain parameter. Within this setting, we first revisit the notion of \emph{robust equilibrium} introduced in \cite{aghassi2006robust}. 
Our goal is to exploit the results of \Cref{sec:mainresult} and bound the probability that an agent will incur a higher cost, compared to what predicted. 

Let $\mcm=\{1,\dots,\N\}$ be a set of agents, where $\j\in\mcm$ is constrained to select $\xj\in\Xj$. Denote $\X\doteq\X^1\times\dots\times\X^\N$. The cost incurred by agent $j\in\mcm$ is described by the function $\Jj(\xj,x^{-\j};\delta) :\X\times\Delta\rightarrow\mb{R}$. 
Since $\Jj$ depends both on the decision of the agents, and on the realization of $\delta\in\Delta$, the notion of Nash equilibrium is devoid of meaning. %
Instead, \cite{aghassi2006robust, crespi2017robust} propose the notion of \emph{robust equilibrium} as a robustification of the former.%
\footnote{A feasible tuple $\xre$ is a robust equilibrium if $\forall j\in\mcm$, $\forall \xj\in\Xj$,  it is 
$\max_{\delta\in\Delta}\Jj(\xre^\j,\xre^{-\j};\delta)\le \max_{\delta\in\Delta}\Jj(\xj,\xre^{-\j};\delta)$, see \cite{aghassi2006robust,crespi2017robust}. }
While a description of the uncertainty set $\Delta$ is seldom available, agents have often access to past realizations  $\{\delta_i\}_{i=1}^N$, which we assume to be independent samples from $(\Delta,\mc{F},\mb{P})$. It is therefore natural to consider the ``sampled'' counterpart of a robust equilibrium.

\begin{definition}[Sampled robust equilibrium]
\label{def:sre}
Given samples $\{\delta_i\}_{i=1}^N$, a tuple $\xsr$ is a sampled robust equilibrium if $\xsr\in\X$ and 
$\max_{i\in\{1,\dots,N\}}\Jj(\xsr^\j,\xsr^{-\j};\delta_i)\le \max_{i\in\{1,\dots,N\}}\Jj(\xj,\xsr^{-\j};\delta_i)$, $\forall \xj\in\Xj$, $\forall j\in\mcm$. 
\end{definition}

\noin 
Observe that $\xsr$ can be thought of as a Nash equilibrium with respect to the worst-case cost functions 
 \be
 \label{eq:worstcosts}
 \Jjmax(x)\doteq \max_{i\in\{1,\dots,N\}} \Jj(x;\delta_i).
 \ee
 In parallel to what discussed in \Cref{subsec:uncertainconstr}, the uncertainty should be regarded as \emph{shared} amongts the agents. 
 In this context, we are interested in bounding the probability that a given agent $j\in\mcm$ will incur a higher cost, compared to what predicted by the empirical worst case $\Jjmax(\xsr)$.%

\begin{definition}[Agent's risk]
\label{def:agentrisk}
The risk incurred by agent $j\in\mcm$ at the given $x\in\mc{X}$ is
\[
V^j(x)=\mb{P}\left\{\delta \in\Delta ~\text{s.t.}~\Jj(x;\delta)\ge 
\Jjmax(x)\right\}
\]
\end{definition}
\noin In addition to existence and uniqueness results, the following corollary provides a bound on such risk measure.

\begin{corollary}[Probabilistic feasibility for $\xsr$]
\label{cor:second}
Assume that, for all $\j\in\mcm$, the cost function $\Jj$ is continuously differentiable, as well as convex in $\xj$ for fixed $x^{-j}$ and $\delta$.
Assume that the sets $\{\Xj\}_{j=1}^\N$ are non-empty, closed, convex.
\begin{itemize}
\item[-] Then, a sampled robust equilibrium exists.
\item[-] Further assume that, for all tuples $(\delta_1,\dots,\delta_N)$, and $N$,
\be
\label{eq:Fmax}
F(x) \doteq 
\begin{bmatrix}
\partial_{x^1} J^1_{\rm max}(x)%
\\
\vdots
\\
\partial_{x^\N} J^{\N}_{\rm max}(x)
\end{bmatrix}
\ee
is strongly monotone.%
\footnote{ $\partial_{x^j} J^j_{\rm max}(x)$ denotes the subgradient of $J^j_{\rm max}$ with respect to $x^j$, computed at $x$. While the operator $F(x)$ is now set valued, the definition of strong monotonicity given in \cref{foot:SMON} can be easily generalized, \cite{boyd2004convex}.}
Then $\xsr$ is unique.
\item[-] Fix $\beta\in (0,1)$. Let $\varepsilon(k)=1-t(k)$, $k\in\mb{N}$, with $t:\mb{N}\rightarrow\mb{R}$ as in Definition \ref{def:t}.
In addition to the previous assumptions, assume that $\xsr$ coincides $\mb{P}^N$ almost-surely with the robust sampled equilibrium of a game obtained by eliminating all the constraints that are not of support.
\mbox{Then, for any agent $j\in\mcm$, any $\Delta$, $\mb{P}$}
\be
\label{eq:probabilisticxsr}
\begin{split}
&\mb{P}^N [V^j(\xsr)\le\varepsilon(\sstar)]\ge 1-\beta
,\\
&\mb{P}^N [V^j(\xsr)\le\varepsilon(n+M)]\ge 1-\beta
,
\end{split}
\ee
where $\sstar$ is the number support constraints of $\xsr$.
\end{itemize}
\end{corollary}
\begin{proof}
See the Appendix.	
\end{proof}
\Cref{cor:second} ensures that, for any given agent $j\in\mcm$, the probability of incurring a higher cost than $%
\Jjmax(\xsr)$ is bounded by $\varepsilon(\sstar)$, with high confidence. 
\section{An application to demand-response markets}
\label{sec:numerics}
In this section, we consider a demand response scheme where electricity scheduling happens 24-hours ahead of time, agents are risk-averse and self-interested.
Formally, given a population of agents $\mcm=\{1,\dots,\N\}$, agent $j\in\mcm$ is interested in the purchase of $\xj_t$ electricity-units at the discrete time $t\in\{1,\dots,T\}$, through a demand-response scheme. Agent $j\in\mcm$ is constrained in his choice to $\xj\in\mc{X}^j\subseteq\mb{R}^T_{\ge0}$  convex, as dictated by its energy requirements. Let $\sigma(x)=\sum_{j=1}^n \xj$ be the total consumption profile.
Given an inflexible demand profile $\md=[d_1,\dots,d_T]\in\mb{R}^{T}_{\ge0}$ corresponding to the non-shiftable loads, the cost incurred by each agent $j$ is given by its total electricity bill
\begin{equation}
\Jj(\xj,\sigma(x);\md)=\sum_{t=1}^T
(\alpha_t\sigma_t(x)+\beta_t d_t) \xj_t,
\label{eq:PEV_energy_bill}
\end{equation}
where we have assumed that, at time $t$, the unit-price of electricity  $c_t\sigma_t(x)+\beta_t d_t$ is a sole function of the shiftable load $\sigma_t(x)$ and of the inflexible demand $d_t$ (with $\alpha_t,\beta_t>0$), in the same spirit of \cite{ma2013decentralized,paccagnan2016aggregative}. In a realistic set-up, each agent has access to a history of previous profiles $\{\md_i\}_{i=1}^N$ (playing the role of $\{\delta_i\}_{i=1}^N$), which we assume to be independent samples from the probability space $(\Delta,\mc{F},\mb{P})$, though $\mb{P}$ is not known.
We model the agents as self-interested and \emph{risk-averse}, so that the notion of sampled robust equilibrium introduced in \Cref{def:sre} is well suited. Assumption \ref{ass:nashexists} is satisfied, while the operator $F$ defined in \eqref{eq:Fmax} is strongly monotone for every $N$ and tuple $(\md_1,\dots,\md_N)$.%
\footnote{This can be seen upon noticing that 
$
\Jjmax(x) = \sum_{t=1}^T(\alpha_t\sigma_t(x))\xj_t +\max_{i\in\{1,\dots,N\}} (B\md_i)^\top \xj,$ where $B=\rm{diag}(\beta_1,\dots,\beta_T)$.
Correspondingly, the operator $F$ is obtained as the sum of two contributions $F=F_1+F_2$. The operator $F_1$ is relative to a game with costs $\{\sum_{t=1}^T(\alpha_t\sigma_t(x))\xj_t\}_{j=1}^\N$, and $F_2$ is relative to a game with costs $\{\max_{i\in\{1,\dots,N\}} (B\md_i)^\top \xj\}_{j=1}^\N$. While $F_1$ has been shown to be strongly monotone in \cite[Lem 3.]{gentile2017nash}, $F_2$ is monotone as it is obtained stacking one after the other the subdifferentials of the convex functions $\{\max_{i\in\{1,\dots,N\}} (B\md_i)^\top \xj\}_{j=1}^\N$. Thus, $F$ is strongly monotone.
}
By \Cref{cor:second}, $\xsr$ exists and is unique. Additionally, under the non-degeneracy assumption, we inherit the probabilistic bounds  \eqref{eq:probabilisticxsr}, whose quality and relevance we aim to test in the following numerics.

We use California's winter daily consumption profiles (available at \cite{EIA19}), as samples of the inflexible demand $\{\md_i\}_{i=1}^N$, on top of which we imagine to deploy the demand-response scheme. In order to verify the quality of our bounds - and only for that reason - we fit a multidimensional Gaussian distribution $\mc{N}(\mu,\Sigma)$ to the data%
. \Cref{fig:samples} displays $100$ samples from the dataset \cite{EIA19} (left), and $100$ synthetic samples from the multidimensional Gaussian model (right). 
\newlength\figureheight 
\newlength\figurewidth 
\setlength\figureheight{3.5cm} 
\setlength\figurewidth{0.38\linewidth} 
\vspace*{-6mm}
\begin{figure}[h!]
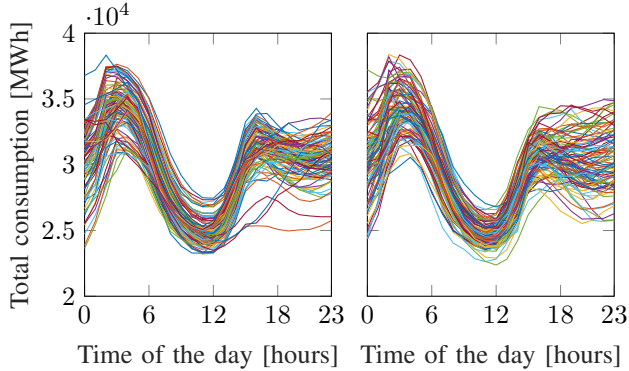

\input{real_samples.tikz}
\hspace*{-3mm}
\input{synthetic_samples.tikz}
\vspace*{-4mm}
\caption{Left: data samples from \cite{EIA19}. Right: synthetic samples $\sim\mc{N}(\mu,\Sigma)$.}
\label{fig:samples}
\vspace*{-3mm}
\end{figure}

We assume that the agents' constraint sets are given by $\mc{X}^j=\{\xj\in\mb{R}_{\ge0}^{24},~\text{s.t.}~ \sum_{t=1}^{24} x^j_t\ge \gamma^j\}$, where $\gamma^j$ is randomly generated according to a truncated gaussian distribution with mean $480$, standard deviation $120$ and $400\le \gamma^j\le 560$, all in MWh. We set $\alpha_t=\beta_t=500\$ / \rm{MWh}^2$, and consider $\N=100$ agents representing, for example, electricity aggregators. We limit ourselves to $N=500$ samples (i.e. a history of $500$ days) to make the example realistic. Since $n+\N=2500$, the a-priori bound in \eqref{eq:probabilisticxsr} is not useful. On the other hand, the values of $\sstar$ observed after extracting $\{\md_i\}_{i=1}^N$ from $\mc{N}(\mu,\Sigma)$ %
and computing the solution $\xsr$, are in the range $3\le \sstar\le 7$. Considering the specific instance with $\sstar=7$, and setting $\beta=10^{-6}$, the a-posteriori bound in \eqref{eq:probabilisticxsr} gives $V^j(\xsr)\le\varepsilon(7)=6.49\%$ %
for all agents, with a confidence of $0.999999$. Since the cost $\Jj(\xj,\sigma(x);\md)$ is linear in $\md$, and $\md\sim\mc{N}(\mu,\Sigma)$, it is possible to compute the risk at the solution $V^j(\xsr)$ in closed form, for each $j\in\mcm$. This calculation reveals that the highest risk over all the agents is $0.16\%\le6.49\%=\varepsilon(7)$, in accordance to \Cref{cor:second} (the lowest value is $0.11\%$).
\setlength\figureheight{3.5cm} 
\setlength\figurewidth{0.38\linewidth} 
\vspace*{-2mm}
\begin{figure}[h!]
\pgfmathdeclarefunction{gauss}{2}{%
  \pgfmathparse{1/(#2*sqrt(2*pi))*exp(-((x-#1)^2)/(2*#2^2))}%
}

\begin{tikzpicture}

\begin{axis}[%
width=\figurewidth,
height=\figureheight,
scale only axis,
domain=3.5e+4:4.5e+4, 
samples=200,
xmin=3.5e+4,
xmax=4.5e+4,
restrict x to domain*=3.5e+4:4.5e+4,
restrict y to domain*=0:1e-3,
xlabel style={font=\color{white!15!black}},
xlabel={Cost [$\$$]},
ymin=0,
ymax=0.0005,
ylabel style={font=\color{white!15!black}},
ylabel={},
legend entries={\footnotesize cost distribution, \footnotesize 
                $\Jjmax(\xsr)$},
legend pos=north west,
legend style={fill=none, draw=none},
legend cell align=left
]
\addplot [very thick, blue] {gauss(39842, 1260.5)};
\addplot [very thick, dashed, color=black]
  table[row sep=crcr]{%
43563.4 	0\\
43563.4 0.000331\\};
\end{axis}
\end{tikzpicture}%
\hspace*{2mm}
\begin{tikzpicture}

\begin{axis}[%
width=\figurewidth,
height=\figureheight,
at={(1.011in,0.642in)},
scale only axis,
xtick={0,6,12,18,23},
xmin=0,
xmax=23,
xlabel style={font=\color{white!15!black}},
xlabel={Time of the day [hours]},
ymin=23000,
ymax=41000,
ylabel style={font=\color{white!15!black}},
ylabel={Total consumption [MWh]},
ylabel style={yshift=-0.4cm}, 
legend style={fill=none, draw=none},
legend cell align=left,
legend pos=north west,
]
\addplot [color=blue, very thick]
  table[row sep=crcr]{%
0	31286.0227814335\\
1	31194.5805369181\\
2	33289.3221476516\\
3	34277.9597315441\\
4	34030.8422818797\\
5	33015.2348993296\\
6	31224.691275172\\
7	31257.7058834469\\
8	31459.2366510213\\
9	31465.767518339\\
10	31725.7181943825\\
11	31837.0420214592\\
12	32065.0589893876\\
13	32151.9672325354\\
14	32122.3307065224\\
15	31482.6806185923\\
16	31012.4563758417\\
17	30900.0805369156\\
18	30662.221476514\\
19	30426.362416113\\
20	30254.0900685837\\
21	30389.3663423973\\
22	30675.8881666974\\
23	31054.1881661025\\
};
\addlegendentry{\footnotesize$\mu+\sigma(\xsr)$}

\addplot [color=black, very thick]
  table[row sep=crcr]{%
0	30077.4060402685\\
1	31194.5805369128\\
2	33289.322147651\\
3	34277.9597315436\\
4	34030.8422818792\\
5	33015.2348993289\\
6	31224.6912751678\\
7	29007.3959731544\\
8	27117.1677852349\\
9	25878.6208053691\\
10	25099.8859060403\\
11	24701.7080536913\\
12	24788.0872483221\\
13	25694.0436241611\\
14	27653.567114094\\
15	30051.9228187919\\
16	31012.4563758389\\
17	30900.0805369128\\
18	30662.2214765101\\
19	30426.3624161074\\
20	30175.9765100671\\
21	30104.855704698\\
22	30248.8691275168\\
23	30458.3993288591\\
};
\addlegendentry{\footnotesize$\mu$}
\end{axis}
\end{tikzpicture}%
\vspace*{-3mm}
\caption{Left: cost distribution for the agent with the highest risk. Right: sum of average inflexible demand $\mu=\mb{E}[\md]$, and flexible demand $\sigma(\xsr)$.}
\label{fig:distribution}
\end{figure}
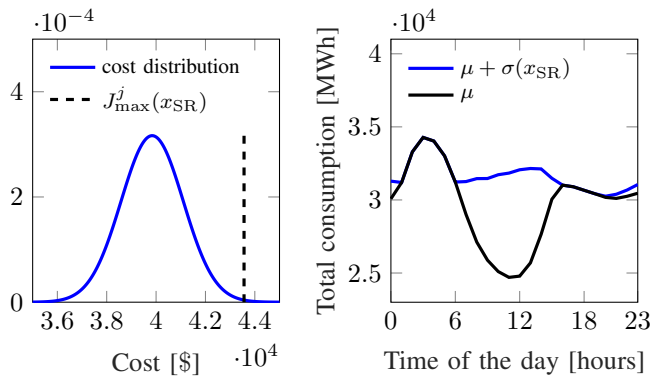

\noindent
 \Cref{fig:distribution} (left) shows the distributions of the cost for the agent with the highest risk.
  \Cref{fig:distribution} (right) shows the sum of the average inflexible demand $\mu$, and the flexible demand $\sigma(\xsr)$.
The difference between $\varepsilon(7)=6.49\%$ and $0.11\%\le V^j(\xsr)\le 0.16\%$, $j\in\mcm$ is partly motivated, by the request that the bound $V^j(\xsr)\le\varepsilon(7)=6.49\%$ holds true for very high confidence $0.999999$.
While an additional source of conservativism  
might be ascribed, at first sight, to having used $V(\xsr)\le \varepsilon(\sstar)$ to derive  $V^j(\xsr)\le \varepsilon(\sstar)$ (see the proof of \Cref{cor:second} in Appendix), this is not the case relative to the setup under consideration. Indeed, Monte Carlo simulations show that  $V(\xsr)\approx 0.17\%$, comparably with $V^j(\xsr)$. In other words, a realization that renders $\xsr$ unfeasible for agent $j$ is also likely to make $\xsr$ unfeasible for agent $l\neq j$. 

\section{Conclusion}
In this manuscript, we aimed at unleashing the power of the scenario approach to the rich class of problems described by variational and quasi variational inequalities. %
As fundamental contribution, we provided a-priori and a-posteriori bounds on the probability that the solution of \eqref{eq:mainproblem} or \eqref{eq:QVI} remains feasible against unseen realizations. We then showed how to leverage these results in the context of uncertain game theory. While this work paves the way for the application of the scenario approach to a broader class of real-world applications, it also  %
generates novel and unexplored research questions.
An example that warrants further attention is that of \emph{tightly} bounding the risk incurred by individual players, when taking data-driven decisions in multi-agent systems.
\appendices
\section{Proof of \Cref{thm:mainresult}}
\label{app:proofmain}
\begin{proof}
The proof of point (i) in \Cref{thm:mainresult} follows the same lines as the proof of \cite[Thm. 4]{campi2018wait} and we here indicate which modifications are needed in the present context.

In particular, the result in \cite[Thm. 4]{campi2018wait} is similar to that presented here in point (i) of \Cref{thm:mainresult}, except that the former refers to an optimization program (see equation (30) in \cite{campi2018wait}) rather than to a variational inequality. More in details, the proof of \cite[Thm. 4]{campi2018wait} follows from that of \cite[Thm. 2 and 3]{campi2018wait}.
Both \cite[Thm. 2 and 3]{campi2018wait} 
hinge upon the result in \cite[Thm. 1]{campi2018wait}. Since all these links survive without any modifications in the present context of VI, any difference must be traced back to the proof of  \cite[Thm. 1]{campi2018wait}.

Turning to the proof of this latter theorem, one sees that it is centered around showing the following fundamental result
\be
\mb{P}^N\{V(\xstar)\ge\varepsilon(\sstar)\}=\sum_{k=0}^n \binom{N}{k}
\int_{(\varepsilon(k),1]}(1-v)^{N-k}dF_k(v),
\label{eq:decompose}
\ee
where $F_k(v)$ is the probability that the tuple $(\delta_1,\dots,\delta_k)$ corresponds to constraints $\{\X_{\delta_i}\}_{i=1}^k$ that are all of support and $V(\xstar_k)\le v$, where $\xstar_k$ denotes the solution obtained from $\{\X_{\delta_i}\}_{i=1}^k$.
This result is originally proven by showing that two events (indicated with $A$ and $B$ in the proof of \cite[Thm. 1]{campi2018wait}) coincide up to a zero probability set. To this purpose, one uses the fact that the solution $x^\star_k$ is feasible for all constraints (a fact that remains valid in our present context when the solution of a VI problem is considered that only contains the first $k$ constraints), and that the support constraints alone return with probability $1$ the same solution that is obtained with all constraints (the so-called non degeneracy condition, which is also valid in the context of the VI problem owing to Assumption \ref{ass:nondegeneracy}). 
Thus, the proof that $A = B$ almost-surely, as presented in \cite{campi2018wait}, does not directly use the fact that $x^\star_k$ is the solution of an optimization program. Instead,  it only uses conditions that hold true also in the present context. Hence, the conclusion extends to the class of variational inequalities.

After having shown that $A = B$ up to a zero probability set, the result in \eqref{eq:decompose} follows using exactly the same argument as that used in \cite{campi2018wait}. The rest of the proof goes through unaltered in the VI context as in \cite[Thm. 1]{campi2018wait}.

\vspace*{\myspace}
Consider now point (ii). We prove that $s^\star \leq n$, from which the result follows from point (i) of the theorem because the function $\varepsilon(k)$ is increasing (see \cite{campi2018wait}). To show that $s^\star \leq n$, let $c = F(x^\star)$, where $\xstar$ is the unique solution to \eqref{eq:mainproblem}, and consider the following optimization problem:
\be
\begin{split}
\label{sample-optimization}
&\min\,c^\top x \nonumber\\
&\,\text{s.t.: } x \in \bigcap_{i=1}^N \mathcal{X}_{\delta_i}. 
\end{split}
\ee
Clearly, $x^\star$ coincides with the solution of \eqref{sample-optimization}, see for example \cite[Eq. 1.3.6]{facchinei2007finite}. Since \cite{calafiore2005uncertain} shows that problem \eqref{sample-optimization} has at most $n$ support constraints, in the following we prove that the support constraints for \eqref{eq:mainproblem} are contained among the support constraints of \eqref{sample-optimization}, from which the result follows. Towards this goal, let us consider a constraint that is not support for \eqref{sample-optimization}. If we remove such constraint, the solution of \eqref{sample-optimization} does not change, which implies that $c^T x \ge c^T x^\star, \forall x \in \tilde{\mathcal{X}}$, where $\tilde{\mathcal{X}}$ is the intersection of all constraints but the removed one; this relation is the same as $F(x^\star)^T (x - x^\star) \ge 0, \forall x \in \tilde{\mathcal{X}}$, showing that $x^\star$ remains the solution of \eqref{eq:mainproblem}, so that the removed constraint is not of support for \eqref{eq:mainproblem}.\footnote{Notice that, differently from problem \eqref{sample-optimization} where the gradient of the cost function is fixed, $F(x)$ depends on $x$ so that removal of a constraint which is not active at $x^\star$ may in principle enlarge the domain and make feasible a point for which a new $\bar{x}^\star$ is such that $F(\bar{x}^\star)^T (x - \bar{x}^\star) \ge 0, \forall x \in \tilde{\mathcal{X}}$. This circumstance, however, is ruled out by the uniqueness requirement in the Standing Assumption.} 
Thus, the set of support constraints of \eqref{eq:mainproblem} is contained in that of \eqref{sample-optimization}. 
\end{proof}
\section*{Proof of \Cref{prop:NEandVI}}
\begin{proof}
The necessary and sufficient condition presented in \cite[Prop. 1.4.2]{facchinei2007finite} for $\xne$ to be a Nash equilibrium coincides with \eqref{eq:mainproblem} where $F$ and $\mc{X}$ are given in \eqref{eq:defFX}, upon noticing that  
\[\small
\begin{split}
	\X^1\times\dots\times \X^{\N}
	&=
\left(\bigcap_{i=1}^N \X_{\delta_i}^1
\right) 
\times
 \left(\bigcap_{i=1}^N \X_{\delta_i}^2
 \right) 
 \times
 \dots 
 \times 
 \left(
 \bigcap_{i=1}^N \X_{\delta_i}^\N
  \right) \\
  &= 
 \bigcap_{i=1}^N \left(\X_{\delta_i}^1\times\dots\times
 \X_{\delta_i}^\N\right)
 =
 \bigcap_{i=1}^N \X_{\delta_i}
 =
 \mc{X}.
\end{split}
\]	
\end{proof}
\section*{Proof of \Cref{cor:first}}
\begin{proof}
\emph{First claim:} thanks to \Cref{prop:NEandVI}, $\xne$ is a Nash equilibrium if and only if it solves \eqref{eq:mainproblem} with $F$ and $\mc{X}$ defined as in \eqref{eq:defFX}. The result follows by applying the existence result of \Cref{prop:suffexun} to the above mentioned variational inequality.
\emph{Second claim:} The result follows thanks to the uniqueness result of \Cref{prop:suffexun} applied to \eqref{eq:mainproblem}.
\emph{Third claim:} this is a direct application of \Cref{thm:mainresult}.
\end{proof}
\section*{Proof of \Cref{cor:second}}
\begin{proof}
Before proceeding with the proof, we recall that $\xsr$ can be regarded as a Nash equilibrium with respect to the costs $\{\Jjmax\}_{j=1}^\N$ defined in \eqref{eq:worstcosts}. 
Thus, in the following we will prove the required statement in relation to $\xsr$ a Nash equilibrium of the game with agents's set $\mcm$, constraint sets $\{\mc{X}^j\}_{j=1}^\N$, and utilities $\{\Jjmax\}_{j=1}^\N$.

\emph{First claim:}  
For each fixed $j\in\mcm$, $\Jjmax$ is continuous in $x$, and convex in $\xj$ for fixed $x^{-j}$, since it is the maximum of a list of functions that are continuous in $x$ and convex in $\xj$. Additionally, each set in $\{\mc{X}^j\}_{j=1}^\N$ is non-empty closed convex.
 Hence, the claim follows by a direct application of the existence result of Debreu, Glicksber, Fan \cite[Thm. 1.2]{91fudenbergtirole}.

\emph{Second claim:}
In parallel to the result of \Cref{prop:NEandVI}, it is not difficult to prove that, under the given assumptions,  $\xsr$ is a sampled robust equilibrium if and only if it satisfies the variational inequality
$F(\xsr)^\top (x-\xsr)\ge0$, $\forall x\in \mc{X}^1\times...\times\mc{X}^\N$, where $F$ is given in \eqref{eq:Fmax}. This can be shown by extending the result of \cite[Prop. 1.4.2]{facchinei2007finite} to the case of subdifferentiable functions $\{\Jjmax\}_{j=1}^\N$, using for example the non-smooth minimum principle of \cite[Thm. 2.1.1]{konnov2001combined}. Uniqueness is then guaranteed by the strong monotonicity of the operator $F$, see \cite[Prop. 2.1.5]{konnov2001combined}.

\emph{Third claim:} in order to show the desired result, we will use an epigraphic reformulation of each agent's optimization problem. %
For ease of presentation, we show the result with $\mc{X}^j=
\mb{R}^m$, but the extension is immediate.

By definition, $\xsr$ is a sampled robust equilibrium if and only if, for each $j\in\mcm$, it holds
$
\xsr^j\in\argmin_{\xj} \Jjmax(\xj,\xsr^{-j}).
$
The latter condition is equivalent to requiring the existence of $\tsr=(\tsr^1,\dots,\tsr^\N)\in\mb{R}^{\N}$ s.t.
\be
\label{eq:GNE}
\begin{split}
(\xsr^j,\tsr^j)\in &\argmin_{(\xj,t^j)}~t^j\\
&~\text{s.t.}~~
t^j\ge \Jj(\xj,\xsr^{-j};\delta_i)\quad\forall\delta_i\in\Delta
\end{split}
\ee
for all $j\in\mcm$, where we have used an epigraphic reformulation of the original problem, and the definition of $\Jjmax$.
Equation \eqref{eq:GNE} can be interpreted as the \emph{generalized Nash equilibrium} condition for a game with $\N$ agents, decision variables $y^j=(x^j,t^j)$, and cost functions  $\{t^j\}_{j=1}^\N$, where each agent $\j\in\mcm$ is given the feasible set 
\[\begin{split}
\mc{Y}^j(y^{-j})&= \bigcap_{i=1}^N\mc{Y}_{\delta_i}^j(y^{-j}),\\
\mc{Y}_{\delta_i}^j(y^{-j})
&\doteq
\{(x^j,t^j)~\text{s.t.}~\Jj(\xj,x^{-j};\delta^i)-t^j\le 0\},
\end{split}
\] 
see \cite[Eq. 1]{facchinei2007generalized}. 
Observe that the cost functions $\{t^j\}_{\j=1}^\N$ are smooth and convex, while each set $\mc{Y}^j(y^{-j})$ is closed convex, for fixed $y^{-j}$ due to the convexity assumptions on $\Jj(\xj,x^{-j};\delta)$. %
Thanks to \cite[Thm. 2]{facchinei2007generalized}, $(\xsr,\tsr)$ is equivalently characterized as the solution of a quasi variational inequality of the form \eqref{eq:QVI} in the augmented space $y=(x,t)\in\mb{R}^{n+\N}$, where 
\[
\begin{split}
F(y)&=
\ones{\N} \otimes
\begin{bmatrix}
\zeros{m}\\
1
\end{bmatrix},\\
\mc{Y}&=\mc{Y}^1(y^{-1})\times\dots\times\mc{Y}^\N(y^{-\N})\\
&=\bigcap_{i=1}^N
\left(
\mc{Y}_{\delta_i}^1(y^{-1})\times\dots\times
\mc{Y}_{\delta_i}^\N(y^{-\N})
\right) = \bigcap_{i=1}^N \mc{Y}_{\delta_i}(y)\,.
\end{split}
\]
Since the latter QVI fully characterizes $\ysr=(\xsr,\tsr)$, and since $\ysr$ is unique (because $\xsr$ is so), the solution to the latter QVI must be unique. Thus, \Cref{thm:mainresultQVI} applies,  and we get
\[
\mb{P}^N [V(\ysr)\le\varepsilon(\sstar)]\ge 1-\beta
~\text{with}~
\varepsilon(\sstar) \doteq 1- t(\sstar),
\]
where $\sstar$ is the number of support constraints of the QVI.
Thanks to \Cref{def:riskQVI} and to the fact that $\tsr^j = \Jjmax(\xsr)$, it is possible to express $V(\ysr)$ as
\[\small
V(\ysr)=\mb{P}\{\delta\in\Delta~\text{s.t.}~ \Jj(\xsr;\delta)\ge \Jjmax(\xsr) \text{ for some }j\}\,.
\]
Thanks to \Cref{def:agentrisk}, it must be $V^j(\xsr)\le V(\ysr)$ from which the claim follows.
\end{proof}
\bibliographystyle{IEEEtran}
\bibliography{biblioVI.bib}
\end{document}